\documentclass[12pt]{amsart}

\usepackage{amsmath,amsthm,amsfonts,amscd,amssymb}
\usepackage{graphicx}
\usepackage{times,euler,eucal,eufrak}  
\usepackage{enumerate}

 \usepackage{hyperref}

\title[Numerical Range]
{Numerical range  for random matrices}

\author[B. Collins, P. Gawron, A. E. Litvak, K. {\.Z}yczkowski]
{Beno\^\i{}t Collins${}^{1}$, Piotr Gawron${}^{2}$,
Alexander E. Litvak${}^{3}$, \\ Karol {\.Z}yczkowski${}^{4,5}$}
\address{
D\'epartement de Math\'ematique et Statistique, Universit\'e d'Ottawa,
585 King Edward, Ottawa, ON, K1N6N5 Canada,
WPI Advanced Institute for Materials Research Tohoku University, Mathematics Unit
2-1-1 Katahira, Aoba-ku, Sendai, 980-8577 Japan
and
CNRS, Institut Camille Jordan Universit\'e  Lyon 1,
France}
 \email{bcollins@uottawa.ca}
\address{Institute of Theoretical and Applied Informatics, Polish Academy
of Sciences, Ba{\l}tycka 5, 44-100 Gliwice, Poland} \email{gawron@iitis.pl}
\address{Dept.~of Math.~and Stat.~Sciences,
         University of Alberta, Edmonton, Alberta, Canada, T6G 2G1,}
\email{aelitvak@gmail.com}
\address{Institute of Physics, Jagiellonian University,
   ul.\ Reymonta 4, 30-059 Krak\'ow, Poland}
\email{karol@tatry.if.uj.edu.pl}
\address{Center for Theoretical Physics, Polish Academy of Sciences,
al.\ Lotnik\'ow 32/46, 02-668 Warszawa, Poland}

 \date{February 28, 2014}

\theoremstyle{plain}
\newtheorem{lemma}{Lemma}[section]
\newtheorem{theorem}[lemma]{Theorem}
\newtheorem{proposition}[lemma]{Proposition}

\theoremstyle{definition}

\theoremstyle{remark}

\newcommand{\Sd}{{\mathbb{S}}}

\newcommand{\E}{{\mathbb{E}}}
\newcommand{\PP}{{\mathbb{P}}}

\newcommand{\C}{{\mathbb{C}}}
\newcommand{\R}{{\mathbb{R}}}

\newcommand\eps{\varepsilon}

\def\r{\right}

\def\lam{\lambda}

\def\imgscale{1}
\begin{document}

\begin{abstract} We analyze the numerical range of high-dimensional random matrices,
obtaining limit results and corresponding quantitative estimates in the non-limit case.
For a large class of random matrices their numerical range is shown to converge to a disc.
In particular, numerical range of complex Ginibre matrix almost surely converges
to the disk of radius $\sqrt{2}$. Since the spectrum of non-hermitian random matrices from the Ginibre
ensemble lives asymptotically in a neighborhood of the unit disk, it follows that the outer
belt of width $\sqrt{2}-1$ containing no eigenvalues can be seen as a quantification the
non-normality of the complex Ginibre random matrix. We also show that the numerical range
of upper triangular Gaussian matrices converges to the same disk of radius $\sqrt{2}$, while
all eigenvalues are equal to zero and we prove that the operator norm of such matrices converges
to $\sqrt{2e}$.
\end{abstract}

\maketitle

\footnotetext[1]{Research partially supported by ERA, NSERC discovery grant, and AIMR.}
\footnotetext[2]{Research partially supported by the Grant N N516 481840
                  financed  by Polish National Centre of Science.}
\footnotetext[3]{Research partially supported by  the
                 E.W.R. Steacie Memorial Fellowship.}
\footnotetext[4]{Research partially supported by the
                 Grant DEC-2011/02/A/ST1/00119 financed by
                 Polish National Centre of Science.}

\section{Introduction}

In this paper we are interested in the numerical range of large random matrices.
In general, {\it the numerical range} (also called {\it the field of values})
of an $N\times N$ matrix is defined as $W(X)=\{ (Xy, y) \, :\,  ||y||_2=1 \}$
(see e.g. \cite{GR97, HJ2, Ki51}).
This notion was introduced almost a century ago
and it is known by the celebrated Toeplitz-Hausdorff theorem
\cite{Hausdorff, Toeplitz}
that $W(X)$ is a compact convex set in  $\mathbb{C}$.
A common convention
 to denote the numerical range by $W(X)$
 goes back to the German term ``Wertevorrat" used by Hausdorff.

For any $N\times N$ matrix $X$ its numerical range $W(X)$ clearly
contains all its eigenvalues $\lambda_i$, $i\leq N$.
If $X$ is normal, that is $XX^*=X^*X$, then its
numerical range is equal to the convex hull of its spectrum,
$W(X)= \Gamma(X):={\rm conv}(\lambda_1, \dots , \lambda_N)$.
The converse is valid if and only if $N \le 4$ (\cite{MoMa, Jo76}).

For a non-normal matrix $X$ its numerical range is typically
larger than $\Gamma(X)$ even  in the case $N=2$. For example,
consider the Jordan matrix of order two,
$$
J_2 = \left[\begin{array}{cc}
                  0  & 1 \\
                  0  & 0 \\
\end{array}\right].
$$
Then both eigenvalues of $J_2$ are equal to zero,
while $W(J_2)$ forms a disk $D(0,1/2)$.

We shall now turn our attention to numerical range of random matrices.
Let $G_N$ be a complex random matrix of order $N$ from the {\sl Ginibre ensemble},
that is an $N\times N$ matrix with
i.i.d centered complex normal entries of variance $1/N$.
It is known that the limiting spectral distribution
$\mu_N$ converges to the uniform distribution on the unit disk
 with probability one
 (cf. \cite{Bai, Gin, Gir, GT, TV1, TV2}).
It is also known that the operator norm goes to $2$ with probability
one. This is directly related to the fact
that the level density of the Wishart matrix $G_N G_N^*$
is asymptotically described by the Marchenko-Pastur law,
supported on $[0,4]$, and the squared largest singular value
of $G_N$ goes to $4$ (\cite{HT03}, see also \cite{Gem} for the real case).

As the complex Ginibre matrix  $G_N$ is generically non-normal,
the support $\Gamma$ of its spectrum is typically smaller than the numerical range $W$.
Our results imply that the ratio between the area of $W(G_N)$ and $\Gamma(G_N)$
converges to $2$ with probability one. Moreover, in the case of strictly upper
triangular matrix $T_N$ with Gaussian entries (see below for precise definitions)
we have that the area of $W(T_N)$ converges to  2, while clearly $\Gamma(T_N)=\{0\}$.

The numerical range of a matrix $X$ of size $N$ can be considered
as a projection of the set of density matrices of size $N$,
$$
  Q_N=\{\varrho: \varrho=\varrho^*, \ \varrho\ge 0, \ {\rm Tr}\varrho=1 \},
$$
onto a plane, where this projection is
given by the (real) linear map ${\rho\mapsto {\rm Tr}\varrho X}$.
More precisely, for any matrix $X$ of size $N$
there exists a real affine rank $2$ projection $P$ of the set ${ Q}_N$,
whose image is congruent to the numerical range $W(X)$ \cite{DGH+11}.

Thus our results on numerical range of random matrices
contribute to the understanding of the geometry of the convex set
of quantum mixed states for large $N$.

Let $d_H$ denotes the Hausdorff distance.
Our main result, Theorem~\ref{mega}, states the following:

\smallskip

\noindent
{\it If random matrices $X_N$ of order $N$ satisfy for every real $\theta$
$$
     \lim _{N\to \infty} \|\mbox{\rm Re }  (e^{i\theta} X_N)\|   = R
$$
then with probability one
$$
  \lim _{N\to \infty} d_H ( W(X_N),  D(0, R)) =0.
$$
}

We apply this theorem to a large class of random matrices. Namely,
let $x_{i,i}$, $i\geq 1$, be i.i.d. complex random variables with
finite second moment,  $x_{i, j}$, $i\ne j$, be i.i.d. centered
complex random variables with finite fourth moment,
and all these variables are independent. Assume $\E |x_{1,2}|^2 = \lambda ^2$
for some $\lambda >0$. Let $ X_N = N^{-1/2}\, \{x_{i, j}\}_{i, j \leq N}, $
and $Y_N$ be the matrix whose entries  above the main diagonal are the same as
entries of $X_N$ and all other entries are zeros. Theorem~\ref{realgin} states
that
$$
   d_H ( W(X_N),  D(0,\sqrt{2}\lambda )) \to 0 \quad \mbox{ and } \quad
   d_H ( W(Y_N),  D(0, \lambda ))\to 0 .
$$
In particular, if $X_N$ is a complex Ginibre matrix $G_N$ or a real
Ginibre matrix $G_N^{\R}$ (i.e. with centered normal entries of variance $1/N$)
and $T_N$ is a strictly triangular matrices $T_N$ with i.i.d centered complex
normal entries of variance $2/(N-1)$  (so that $\E \mbox{Tr} X_N X_N^* =
\E \mbox{Tr} T_N T_N^* =N$) then  with probability one
$$
   d_H ( W(G_N),  D(0,\sqrt{2} )) \to 0 \quad \mbox{ and } \quad
   d_H ( W(T_N),  D(0,\sqrt{2} ))\to 0 .
$$

We also provide corresponding quantitative estimates on the rate of the convergence
in the case of $G_N$ and $T_N$.

A related question to our study is the
limit behavior of the operator (spectral)
norm $\|T_N\|$ of a random triangular matrix,
which can be used to characterize its non-normality.
As we mentioned above, it is known that with probability one
\begin{equation}\label{normlim}
 \lim_{N\to \infty} \|G_N\| =  2  .
\end{equation}
It seems that the limit behavior of $\|T_N\|$ has not been investigated yet,
although its limiting counterpart has been extensively studied
by Dykema and Haagerup in the framework of investigations around
the invariant subspace problem.
In the last section (Theorem~\ref{trinor}), we prove that   with probability one
\begin{equation}\label{trnorm}
 \lim_{N\to \infty} \|T_N\| =  \sqrt{2 e} .
\end{equation}
Note that in Section~\ref{secttri} this fact is formulated and proved
in another normalization.

Our proof here is quite indirect and relies on strong convergence for
random matrices established by \cite{HTh}. In particular, our proof does
not provide any quantitative estimates for the rate of convergence.
It would be interesting to obtain corresponding deviation inequalities.
We would like to mention that very recently the empirical eigenvalue measures
for large class of symmetric random matrices of the form $X_N X_N^*$, where
$X_N$ is a random triangular matrix, has been investigated (\cite{LP}).

The paper is organized as follows. In Section~\ref{illustr}, we provide some
preliminaries and numerical illustrations.  In Section~\ref{basic}, we provide
basic facts on the numerical range and on the matrices formed using Gaussian random
variables. The main section, Section~\ref{mainresults}, contains the results on convergence
of the numerical range of random matrices mentioned above (and the corresponding quantitative
estimates). Section~\ref{concrem} suggests a possible extension of the main theorem, dealing
with a more general case, when the limit of
$\|\mbox{\rm Re }  (e^{i\theta} X_N)\|$ is a (non-constant) function of $\theta$. Finally, in
Section~\ref{secttri}, we provide the proof of (\ref{trnorm}).

\section{Preliminaries and numerical illustrations}
\label{illustr}

By $\xi$, we will denote a centered complex Gaussian random variable, whose variance may change
from line to line.  When (the variance of) $\xi$ is fixed,
$\xi_{ij}$, $i,j\geq 1$ denote independent copies of $\xi$.
Similarly, by $g$  we will denote a centered real Gaussian random variable, whose variance may change
from line to line.  When (the variance of) $g$ is fixed,
$g_{ij}$, $i,j\geq 1$ denote independent copies of $g$.

We deal with  random matrices $X_N$ of size $N$.
To set the scale we are usually going to normalize random matrices
by fixing their expected Hilbert-Schmidt norms to be equal to $N$, i.e.
$\E \|X_N\| _{\rm HS}^2=\E {\rm Tr} X_N X_N^*=N$.
We study the following ensembles.

\begin{enumerate}
\item \label{enscomgin}
Complex Ginibre matrices $G_N$ of order $N$ with entries $\xi_{ij}$, where
$\E |\xi _{ij}|^2 =1/N$. As we mention in the introduction, by the  circular law,
the spectrum of $G_N$ is asymptotically contained in the unit disk.
Note $\E \|G_N\| _{\rm HS}^2=N$.

\item \label{ensrealgin}
Real Ginibre matrices $G_N^{\mathbb{R}}$  of order $N$ with entries $g_{ij}$, where
$\E |g _{ij}|^2 =1/N$.
Note $\E \|G_N^{\mathbb{R}}\| _{\rm HS}^2=N$.

\item  \label{enstri}
Upper triangular random matrices $T_N$  of order $N$ with entries $T_{ij}=\xi_{ij}$ for $i <j$
    and $T_{ij}=0$ elsewhere, where $\E |\xi _{ij}|^2 =2/(N-1)$.
Clearly, all eigenvalues of $T_N$  equal to zero.
Note $\E \|T_N\| _{\rm HS}^2=N$.

\item  \label{ensdiag}
 Diagonalized Ginibre matrices, $D_N = Z G_N Z^{-1}$  of order $N$, so that
   $D_{kl}=\lambda_k \delta_{kl}$ where $\lambda_k$, $k=1,\dots, N$,
   denote complex eigenvalues of  $G_N$. Note that $G_N$ is diagonalizable
   with probability one. In order to ensure the uniqueness of the probability
   distribution on diagonal matrices,
   we assume that it is invariant under conjugation by permutations.
  Note that  integrating over the Girko circular law
  one gets the average squared eigenvalue of the complex Ginibre matrix,
  $\langle |\lambda|^2\rangle =\int_{0}^1 2x^3 dx=1/2$. Thus,
 $\E \|D_N\| _{\rm HS}^2=N/2$.

\item \label{ensunit}
   Diagonal unitary matrices $U_N$  of order $N$ with entries $U_{kl}=\exp(i \phi_k) \delta_{kl}$,
   where $\phi_k$ are independent uniformly distributed on $[0, 2 \pi)$ real random
   variables.

\end{enumerate}


The structure of some of these matrices is exemplified below for the case $N=4$.
Note that the variances of $\xi$ are different in the case of  $G_4$ and
in the case of $T_4$. To lighten the notation they are depicted by the
same symbol $\xi$,  but  entries are independent.

$$
G_4 = \left[\begin{array}{cccc}
                  \xi  & \xi & \xi & \xi \\
                  \xi  & \xi & \xi & \xi \\
                  \xi  & \xi & \xi & \xi \\
                  \xi  & \xi & \xi & \xi \\
\end{array}\right], \
T_4 = \left[\begin{array}{cccc}
                    0  & \xi & \xi &  \xi \\
                    0  &  0  & \xi &  \xi \\
                    0  &  0  &  0  &  \xi \\
                    0  &  0  &  0  &   0  \\
\end{array}\right], \
 D_4  = \left[\begin{array}{cccc}
                    \lambda_1 & 0  &   0  &   0 \\
                    0  & \lambda_2 &   0  &   0 \\
                    0  &  0  &  \lambda_3 &   0 \\
                    0  &  0  &   0  &  \lambda_4 \\
\end{array}\right] .
\label{struct}
$$



We  will study the following parameters of  a
given (random) matrix $X$:

\begin{enumerate}[(a)]

\item
 {\sl the numerical radius}
$r(X)= {\rm max}\{|z|: z \in W(X)\}$,
\item
{\sl the spectral radius} $\rho(X)=  |\lambda_{\max}|$,
where $\lambda_{max}$ is the leading eigenvalue of $X$
with the largest modulus,

\item
 {\sl the operator (spectral) norm} equal to the
largest singular value, $\|X\|=\sigma_{\max}(X)=\sqrt{\lambda_{max} (XX^*)}$
(and equals to the operator norm of $X$, considered as an operator
$\ell_2^N \to \ell_2^N$),
\item {\sl the non-normality measure}
$\mu_3(X):=(||X||^2_{\rm HS} - \sum_{i=1}^N |\lambda_i|^2)^{1/2}$.

\end{enumerate}

The latter quantity, used by Elsner and Paardekooper \cite{EP87},
is based on the Schur lemma:
As the squared Hilbert-Schmidt norm of a matrix
can be expressed by its singular values,
 $||X||^2_{\rm HS}=\sum_{i=1}^N \sigma_i^2$,
the measure $\mu_3$ quantifies the difference between
the average squared singular value and the average squared absolute
value of an eigenvalue, and vanishes for normal matrices.
Comparing the expectation values for the squared norms
of a random Ginibre matrix $G_N$ and a diagonal matrix $D_N$
containing their spectrum we establish the following statement.

The squared non-normality coefficient $\mu_3$  for a
complex Ginibre matrix $G_N$ behaves asymptotically as
\begin{equation}
\label{mu3gini}
\mathbb{E} \mu_3^2(G_N) = \mathbb{E} \|G_N\|^2_{\rm HS} - \mathbb{E} \|D_N\| _{\rm HS}^2 = N/2.
\end{equation}
Since all eigenvalues of random triangular matrices are equal to zero
an analogous results for the ensemble
of upper triangular random matrices reads
 $\mathbb{E} \mu_3^2(T_N) = N$.


Figure \ref{fig1} shows the numerical range of the complex Ginibre matrices
of ensemble (\ref{enscomgin}), which tends asymptotically to the disk of radius
$\sqrt{2}$ -- see Theorem~\ref{realgin}.
As the convex hull of the spectrum, $\Gamma(G_N)$, goes to
the unit disk, the ratio of their area tends to $2$
and characterizes the non-normality of a generic Ginibre matrix.
By the non-normality belt we mean the set difference $W(X)\setminus \Gamma(X)$,
which contains no eigenvalues.

As  $N$ grows to infinity,
spectral properties of the real Ginibre matrices of ensemble (\ref{ensrealgin})
become analogous to the complex case.
By Theorem~\ref{realgin},
in both cases numerical range converges to $D(0, \sqrt{2})$
and the spectrum is supported by the unit disk.
The only difference is the symmetry of the
spectrum with respect to the real axis
and a clustering of eigenvalues along the real axis
for the real case.

\begin{figure}[htbp]
\begin{center}
\includegraphics[scale=\imgscale]{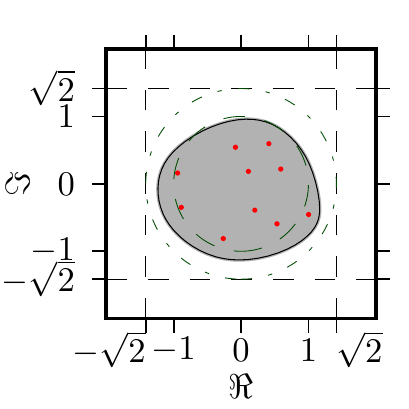}
\includegraphics[scale=\imgscale]{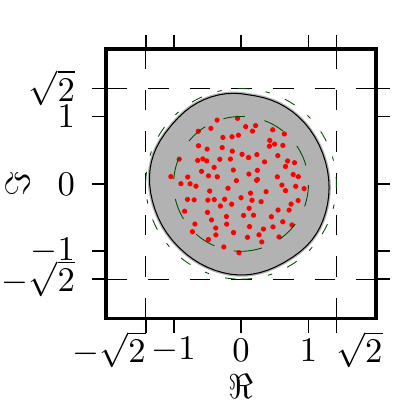}
\includegraphics[scale=\imgscale]{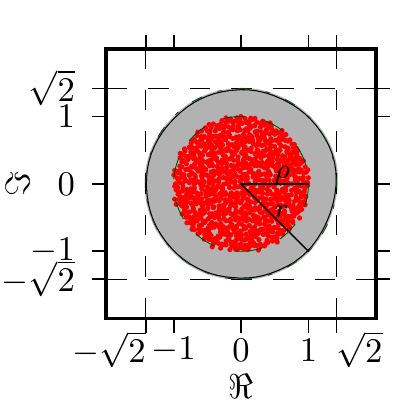}
\caption{Spectrum (dots) and numerical range (dark convex set)
of the complex  Ginibre matrices of sizes  $N=10, 100$ and $1000$.
The spectrum is asymptotically contained in the unit disk
while numerical range converges to a disk of radius $r=\sqrt{2}$
denoted in the figures. Note the outer ring of the range is
the non-normality belt of width $\sqrt{2}-1$ (it contains no
eigenvalues).
}
\label{fig1}
\end{center}
\end{figure}

Figure \ref{fig2} shows analogous examples
of diagonal matrices $D$ with the Ginibre spectrum -- ensemble (\ref{ensdiag}).
Diagonal matrices are normal, so the numerical range equals to the
support of the spectrum and thus converges to the unit disk.
Note that this property hold also for a
``normal Ginibre ensemble"
of matrices of the kind $G'=VDV^*$,
where $D$ contains the spectrum of a Ginibre matrix,
while $V$ is a random unitary matrix
drawn according to the Haar measure.

\begin{figure}[htbp]
\begin{center}
\includegraphics[scale=\imgscale]{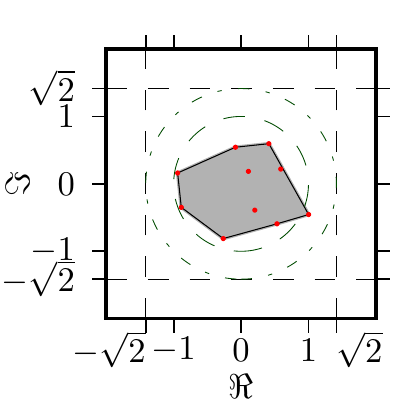}
\includegraphics[scale=\imgscale]{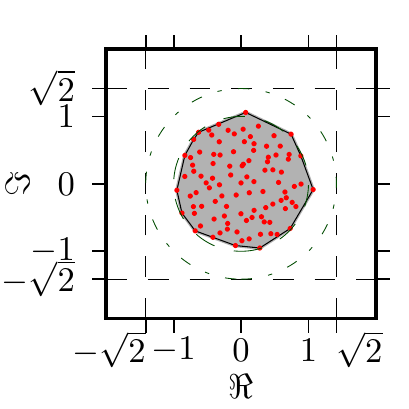}
\includegraphics[scale=\imgscale]{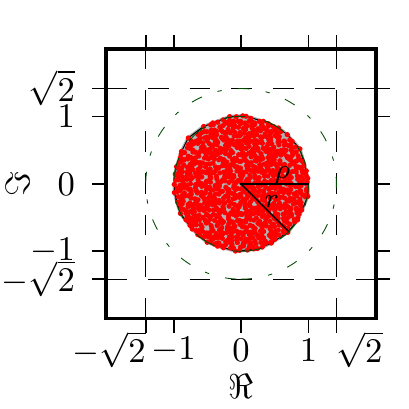}
\caption{As in Fig. \ref{fig1},  for ensemble of
diagonal matrices $D_N$ containing spectrum of Ginibre matrices
of sizes  $N=10,100$ and $1000$.
Numerical range of these normal matrices coincides with
the convex hull of their spectrum.
}
\label{fig2}
\end{center}
\end{figure}

Analogous results for the  upper triangular matrices $T$ of ensemble
(\ref{enstri}) shown in Fig.\ref{fig3}. The numerical range asymptotically
converges to the disk of radius $\sqrt{2}$ with probability one -- see
Theorem~\ref{realgin}.

As all eigenvalues of $T$ are zero, 
the asymptotic properties of the spectrum and numerical
range of $T$ become identical with these of a Jordan matrix $J$
of the same order $N$ rescaled by $\sqrt{2}$.
By construction $J_{km}=1$ if $k+1=m$ and zero elsewhere
for $k,m=1,\dots , N$. It is known \cite{Wu98} that numerical range
of a Jordan matrix $J$ of size $N$ converges to the unit disk as
$N \to \infty$.

\begin{figure}[htbp]
\begin{center}
\includegraphics[scale=\imgscale]{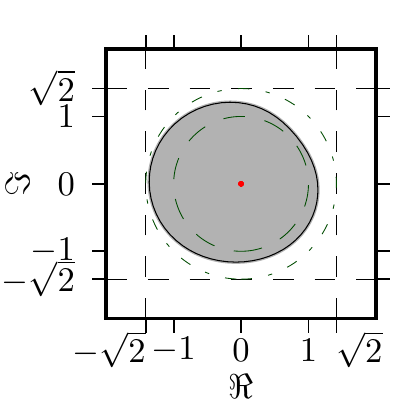}
\includegraphics[scale=\imgscale]{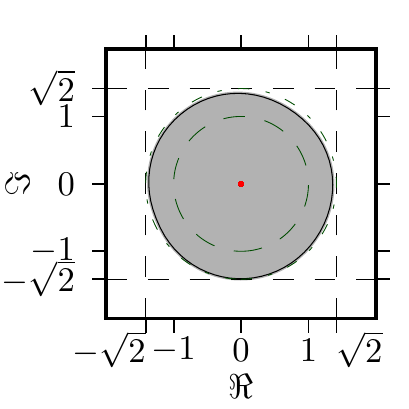}
\includegraphics[scale=\imgscale]{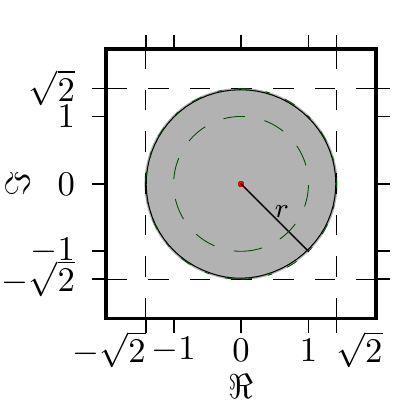}
\caption{As in Fig. \ref{fig1}, for
upper triangular  random matrices  $T_N$ of sizes $N=10,100$ and $1000$,
for which all eigenvalues are equal to zero and the
numerical range converges to the disk of radius $\sqrt{2}$.
}
\label{fig3}
\end{center}
\end{figure}

In the table below we listed asymptotic predictions
for the operator (spectral) norm $\|X\|$, the numerical radius $r(X)$, the
spectral radius $\rho(X)$ and the squared non-normality parameter,
$\bar{\mu}_3^2= \mathbb{E}(\mu_3^2)$,
of generic matrices pertaining to the ensembles investigated.

\bigskip
\begin{center}
{\renewcommand{\arraystretch}{1.4}
\begin{tabular}[c]{ c c c c c }\hline
Ensemble  &  \quad $\|X\|$ \quad & \quad  $r(X)$ \quad
& \quad  $\rho(X)$ \quad  & \quad  $\bar{\mu}^2 _3(X)$ \quad \\
\hline
 Ginibre    $G$  & $2$        & $\sqrt{2}$ & $1$ & $  N/2$ \\
 Diagonal   $D$  & $1$        & $1$        & $1$ & $0$          \\
 Triangular $T$  & $\sqrt{2 e}$ & $\sqrt{2}$        & $0$ & $ N$   \\
\hline
\end{tabular}
} 
\end{center}

\bigskip

Consider a matrix $X$ of order $N$, normalized as Tr$XX^*=N$.
Assume that the matrix is diagonal,
so that its numerical range $W(X)$ is formed by the convex hull of the diagonal entries.
Let us now modify the matrix $X$,
writing $Y=\sqrt{1-a}X+\sqrt{a}T$,
where $T$ is a strictly upper triangular random matrix normalized as above and
$0 \le a \le 1$. Note Tr$YY^*=N$ as well.
Rescaling $X$ by a number $\sqrt{1-a}$
smaller than one
and adding an off-diagonal part $\sqrt{a}T$
increases the non-normality belt of $Y$,
i.e. the set $W(Y)\setminus \Gamma(Y)$.
The larger relative weight of the off-diagonal part,
the larger squared non-normality index,
$\mu_3^2(Y)= \|Y\|^2_{\rm HS} - \sum_{i=1}^N |Y_{ii}|^2= N-(1-a)N=aN$
and the larger the non-normality belt of the numerical range.
In the limiting case  $a\to 1$ the off-diagonal part $\sqrt{a}T$
dominates the matrix $Y$. In particular, if $T=T_N$ of ensemble (\ref{enstri})
then its numerical range converges to
the disk of radius $\sqrt{2}$ as $N$ grows to infinity.

To demonstrate this construction in action we plotted in
Fig. \ref{fig4} numerical range of an exemplary random matrix
$Y'=D_N+\frac{1}{\sqrt{2}}T_N$, which contains the spectrum of
the complex Ginibre matrix $G_N$ at the diagonal, and the matrix $T_N$
in its upper triangular part. The relative weight $a=1/\sqrt{2}$
is chosen in such a way that Tr$Y'{Y'}^*=N$. Thus $Y'$ displays similar properties
to the complex Ginibre matrix: its numerical range is close to a disk of
radius $r=\sqrt{2}$,
while the support of the spectrum is close to the unit disk.
This observation is related to the fact  \cite{Mehta}
that bringing the complex Ginibre matrix by a unitary rotation to its triangular
Schur form, $S:=UGU^*=D+T$, one assures that
the diagonal matrix $D$ contains spectrum of $G$,
while $T$ is an upper triangular matrix
containing independent Gaussian random numbers.

\begin{figure}[htbp]
\begin{center}
\includegraphics[scale=\imgscale]{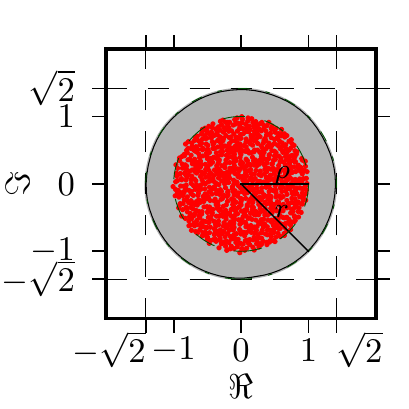}
\includegraphics[scale=\imgscale]{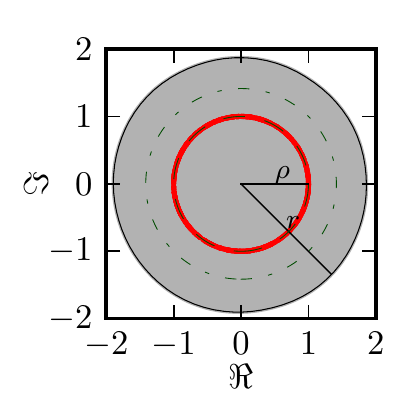}
\caption{As in Fig. \ref{fig1}, for
a) $D_N+\frac{1}{\sqrt{2}}T_N$ and
b) $U_N+T_N$ of size $N=1000$.
}
\label{fig4}
\end{center}
\end{figure}

Another illustration of the non-normality belt is presented in
Fig. \ref{fig4}b. It shows the numerical range of the sum of a
diagonal random unitary matrix $U_N$ of ensemble (\ref{ensunit}),
with all eigenphases drawn independently according to a uniform distribution,
with the upper triangular matrix $T_N$ of ensemble (\ref{enstri}). All eigenvalues
of this matrix belong to the unit circle, while presence of
the triangular contribution increases the numerical radius $r$
and forms the non-normality belt.
Some other examples of numerical range computed
numerically for various ensembles of random matrices can be found
in \cite{Paw13}.

\section{Some basic facts and notation}
\label{basic}

In this paper, $C_0$, $C_1$, ..., $c_1$, $c_2$, ... denote absolute positive constants, whose
value can change from line to line.
Given a square matrix $X$, we denote
$$
    \mbox{Re } X = \frac{X+X^*}{2} \quad \quad \mbox{ and } \quad \quad  \mbox{Im } X = \frac{X-X^*}{2i},
$$
so that $X= \mbox{Re } X + i \mbox{ Im } X$ and both $\mbox{Re } X$ and $\mbox{Im } X$ are self-adjoint matrices. Then it
is easy to see that
$$
     \mbox{Re } W(X) = W( \mbox{Re } X) \quad \quad \mbox{ and } \quad \quad \mbox{Im }  W(X) = W(\mbox{Im } X).
$$

Given $\theta \in [0, 2\pi]$, denote $X_{\theta}:= e^{i\theta }X$ and  by $\lam_{\theta}$ denote the
maximal eigenvalue of $\mbox{Re } X_{\theta}$. It is known (see e.g. Theorem~1.5.12 in \cite{HJ2}) that
\begin{equation} \label{support}
  W(X) = \bigcap_{0\leq \theta \leq 2\pi} H_{\theta},
\end{equation}
where
$$
   H_{\theta} = e^{-i\theta} \left\{ z\in \C \, : \, \mbox{Re } z \leq  \lam _\theta  \right\}.
$$
Our results for random matrices are somewhat similar, however we use the norm $\|X_{\theta}\|$
instead of its maximal eigenvalue. Repeating the proof of (\ref{support})
(or adjusting the proof of Proposition~\ref{inclusion} below), it is not difficult
to see that
\begin{equation} \label{inclus}
   W(X) \subset K(R),
\end{equation}
where $K(R)$ is a star-shaped set defined by
\begin{equation} \label{radial}
 K(R) := \{ \lambda e^{- i\theta}\, \|X_{\theta}\| \, \, : \, \, \lambda \in [0, 1],\,  \theta \in [0, 2\pi)\}.
\end{equation}
Below we provide a complete proof of corresponding results for random matrices.
Note that $K(R)$ can be much larger than $W(X)$. Indeed, in the case of the identity operator
$I$ the numerical range is a singleton, $W(I)=\{1\}$, while the set $K(R)$ is defined by the
equation $\rho \leq |\cos t|$ (in the polar coordinates).

%
%
%

\subsection{GUE}
\label{GUE}

We say that a Hermitian $N\times N$ matrix $A=\{A_{i,j}\}_{i,j}$ pertains to
Gaussian Unitary Ensemble (GUE) if {\bf a.} its entries $A_{i,j}$'s are independent
for $1\leq i\leq j\leq N$, {\bf b.} the entries $A_{i,j}$'s for $1\leq i< j\leq N$
are complex  centered Gaussian random variables of variance $1$
(that is the real and imaginary parts are independent centered Gaussian of variance
$1/2$), {\bf c.} the entries $A_{i,i}$'s for $1\leq i\leq N$ are real centered
Gaussian random variables of variance $1$.

Clearly, for the complex Ginibre matrix $G_N$ its real part,
$Y_N:= Re(G_N)$, is a $(2N)^{-1/2}$ multiple of a GUE. It is known that
with probability one $\|Y_N\| \to \sqrt{2}$
(see e.g. Theorem~5.2 in \cite{BS} or Theorem~5.3.1 in \cite{PS}).
We will also need the following quantitative estimates.
In \cite{Au, Le1, Le2, LR} it was shown that
for GUE, normalized as  $Y_N$, one has for every $\eps\in (0, 1]$,
$$
  \PP \left( \|Y_N\| \geq \sqrt{2} +\eps \r) \leq C_0 \exp(-c_0 N \eps ^{3/2}).
$$
Moreover, in \cite{LR} it was also shown that for $\eps\in (0, 1]$,
$$
  \PP \left( \|Y_N\| \leq \sqrt{2} - \eps \r) \leq C_1 \exp(-c_1 N^2 \eps ^{3}).
$$
Note that $C_1 \exp(-c_1 N^2 \eps ^{3})\leq C_2 \exp(-c_1 N \eps ^{3/2})$.
Thus, for $\eps\in (0, 1]$,
\begin{equation}\label{maxev}
  \PP \left( |\|Y_N\| - \sqrt{2}| > \eps \r) \leq  C_3 \exp(-c_2 N \eps ^{3/2})
\end{equation}
(cf. Theorem~2.7 in \cite{DS}).
It is also well known (and follows from concentration) that  there exists two
absolute constants $c_4$ and $C_4$ such that
\begin{equation}\label{norm}
  \PP \left( \|G_N\|  \geq 2.1 \r) \leq  C_4 \exp(-c_4 N) .
\end{equation}

%
%
%

\subsection{Upper triangular matrix}
\label{uptrima}

Let $g_i$, $h_i$, $i\geq 1$, be independent $\mathcal{N}(0,1)$ real random variables.
It is well-known (and follows from the Laplace transform) that
$$
  \mathbb{E} \max _{i\leq N} |g_i| \leq \sqrt{2\ln (2N)}.
$$
Since $\|x\|_{\infty}\leq \|x\|_2$, the classical Gaussian concentration
inequality (see \cite{CIS} or inequality (2.35)  in \cite{Le}) implies that
for every $r>0$,
\begin{equation}\label{maxgauss}
   \PP \left(\max _{i\leq N} |g_i| > \sqrt{2\ln (2N)} + r \right)\leq e^{-r^2/2}.
\end{equation}

Recall that $T_N$ denotes the upper triangular $N\times N$ Gaussian random matrix normalized
such that $\mathbb{E} T_N T_N^* =N$, that is $(T_N)_{ij}$ are independent complex
Gaussian random variables of variance $2/(N-1)$ for $1\leq i<j\leq N$ and $0$ otherwise.
Note that $\mbox{Re } T_N$ can be presented as $Z_N/\sqrt{2(N-1)}$, where
 $Z_N$ is a complex Hermitian $N\times N$ matrix with zero on the diagonal and
independent complex Gaussian random variables of variance one   above the diagonal.
Let $A_N$ be distributed as GUE (with $g_i$'s on the diagonal) and $V_N$ be the diagonal matrix
with the same diagonal as $A_N$.  Clearly, $Z_N=A_N-V_N$. Therefore, the triangle inequality
and  (\ref{maxev}) yield that for every $\eps\in (0, 1]$
\begin{equation}\label{devtrig}
  \PP \left( \left|\frac{1}{\sqrt N} \|Z_N\| - 2 \right| > \eps \r) \leq  C \exp(-c N \eps ^{3/2}) ,
\end{equation}
where $C$ and $c$ are absolute positive constants (formally, applying the triangle inequality,
 we should ask $\eps > \sqrt{\ln(2N)/N}$,
but if $\eps \leq \sqrt{\ln(2N)/N}$  the right hand side becomes large than 1,
by an appropriate choice of the constant $C$). In particular, the Borel-Cantelli lemma
implies that with probability one $\|Z_N\|/\sqrt{N} \to  2$
(alternatively one can apply Theorem~5.2 from \cite{BS}).


\section{Main results}
\label{mainresults}

Our first main result is

\begin{theorem}\label{mega}
Let $R>0$. Let $\{X_N\}_N$ be a sequence of complex random $N\times N$  matrices
such that for every $ \theta\in \R$  with probability one
$$
     \lim _{N\to \infty} \|\mbox{\rm Re }  (e^{i\theta} X_N)\|   = R.
$$
Then with probability one
$$
  \lim _{N\to \infty} d_H \left( W(X_N),  D(0, R)\right) =0.
$$
Furthermore, if there exists $A\geq  \max\{R, 1\}$ such that for every $N\geq 1$,
$$
  \PP\left( \|X_N\| > A \right) \leq p_N
$$
and for every $\eps\in (0,1/2)$, $N\geq 1$, $\theta\in \R$,
$$
   \PP\left( \left|\, \|\mbox{\rm Re }  (e^{i\theta} X_N)\| - R\right| > \eps \right) \leq q_N (\eps)
$$
then for every positive $\eps \leq \min\{1/2, \sqrt{R/(A+1)}\}$ and every $N$ one has
$$
  \PP\left( d_H(W(X_N), D(0, R) > 4 A \eps \r)  \leq
   p_N + 7 R \eps^{-2} \   q_N (\eps^2) .
$$
\end{theorem}


\begin{proof}
Fix positive $\eps \leq \min\{1/2, R/(A+1)\}$. Since the real part of a matrix is a
self-adjoint operator we have
$$
  \lam(\theta, N) :=\| \mbox{Re } (e^{i\theta} X_N )\| = \sup\{ \mbox{Re }
  (e^{i\theta} X_N y, y)\, : \,   \|y\|_2=1\}.
$$
By assumptions of the theorem,  for every $\theta \in \R$
 with probability one
$$
   \lim _{N\to \infty } \lam(\theta, N) = R.
$$

Let $\Sd$ denote the boundary of the disc $D(0, R)$.
Choose a finite  $\eps$-net $\mathcal{N}$ in $[0, 2\pi]$, so that $\{R e^{i\theta}\}_{\theta\in {\mathcal{N}}}$
is an $\eps$-net (in the geodesic metric) in  $\Sd$.
Then, with probability one, for every $\theta \in {\mathcal{N}}$ one has
$\lam(\theta, N)\to R$.

 Since $\mbox{Im } X_N = \mbox{Re } (e^{-i \pi/2}X_N)$, one has
$$
 R \leq  \limsup _{N\to \infty} \|X_N\| \leq \limsup _{N\to \infty} \|\mbox{Re } X_N\|
  + \limsup _{N\to \infty} \|\mbox{Im } X_N\| = 2R .
$$
 Choose $A\geq  \max\{R, 1\}$ and  $N\geq 1$ such that  for every $M\geq N$ one has
$$
  \|X_M\| \leq  A \quad \mbox{ and } \quad \forall  \theta \in {\mathcal{N}} \, \, \, \,
   |\lam(\theta, M) - R| \leq \eps .
$$
Fix $M\geq N$.
Note that the supremum in the definition of $\lam(\theta, M)$ is attained
and that
$$
  |\mbox{Re } (e^{i\theta }X_M y, y) - \mbox{Re } (e^{i t }X_M y, y)| \leq
  |e^{i\theta } - e^{i t }|\cdot  |(X_M y, y)|\leq  \eps  A,
$$
whenever $|\theta - t|\leq \eps$ and $\|y\|_2=1$.
Using approximation by elements of $\mathcal{N}$, we obtain for every real $t$,
$$
        |\lam(t, M) - R |\leq (A + 1) \eps .
$$

Let $y_0$ be such that $\|y_0\|=1$ and
$$
  \lam := \sup\{|(X_M y, y)| \, : \,  ||y||_2=1\} =  |(X_M y_0, y_0)|.
$$
Then for some $t$
$$
  \lam = e^{i t} (X_M y_0, y_0) = \mbox{Re } (e^{i t}  X_M y_0, y_0)=
  \lam (t, M) \leq R + (A + 1) \eps  .
$$
This shows that $W(X_M) \subset D(0,  R + (A + 1) \eps)$.

Finally fix some $z\in \Sd$, that is $z=  R e^{i t}$. Choose $\theta \in {\mathcal{N}}$
such that $|t-\theta|\leq \eps$. Let $y_1$ be such that
$$
  \lam (-\theta, M) = \mbox{Re } (e^{-i\theta }X_M y_1, y_1) =
  \mbox{Re } (e^{-i\theta } (X_M y_1, y_1)).
$$
Denote $x:= (X_N y_1, y_1)$. Then
$$
  R - (A + 1) \eps \leq \mbox{Re } (e^{-i\theta } x) \leq |x| \leq  R + (A + 1) \eps .
$$
 Since $A\geq \max\{R, 1\}$ and $\eps\leq R/(A+1)$, this implies that
$$
  |R e^{i\theta }-x| \leq \sqrt{(A + 1)^2 \eps^2 + 4 R (A + 1) \eps } \leq
  2 A \sqrt{\eps} \, \sqrt{\eps+2}.
$$
 Since $|t-\theta|\leq \eps$ and $\eps <1/2$, we observe that
$$
  |z-x|\leq  R |e^{i t } -e^{i\theta }| + |R e^{i\theta }-x|
  \leq  R \eps + 2\sqrt{2.5} A \sqrt{\eps} \leq 4 A \sqrt{\eps}.
$$
Therefore, for every $z\in \Sd$ there exists  $x\in W(X_M)$ with
$$
  |z-x| \leq 4 A \sqrt{\eps}.
$$

Using convexity of $W(X_M)$, we obtain that with probability one
$$
  d_H (W(X_M), D(0, R)) \leq  4 A \sqrt{\eps}.
$$
Since $M\geq N$ was arbitrary, this implies the desired result.

The proof of the second part of the theorem is essentially the same.
Note that the $\eps$-net in our  proof can be chosen to have the cardinality
not exceeding $2.2\pi R/\eps$.
Thus, by the union bound,
the probability of the event
$$
  \|X_M\| \leq  A \quad \mbox{ and } \quad \forall  \theta \in {\mathcal{N}} \, \, \, \,
   |\lam(\theta, M) - R | \leq \eps ,
$$
considered  above,
does not exceed
$p_N + 2.2 \pi R \ \eps^{-1} q_N (\eps).$
This implies the quantitative part of the theorem.
\end{proof}

\bigskip

The next theorem shows that the first part of Theorem~\ref{mega} applies to a large class
of random matrices (essentially to matrices whose entries are i.i.d. random variables
having final fourth moments and corresponding triangular matrices), in particular to ensembles
$G_N$, $G^{\R}_N$ and $T_N$ introduced in Section~\ref{illustr}.



\begin{theorem}\label{realgin}
Let $x_{i,i}$, $i\geq 1$, be i.i.d. complex random variables with finite second moment,
$x_{i, j}$, $i\ne j$, be i.i.d. centered complex random variables with finite fourth moment,
and all these variables are independent. Assume $\E |x_{1,2}|^2 = \lambda ^2$
for some $\lambda >0$. Let $ X_N = N^{-1/2}\, \{x_{i, j}\}_{i, j \leq N}, $
and $Y_N$ be the matrix whose entries on or above the diagonal are the same as
entries of $X_N$ and entries below diagonal are zeros.
Then with probability one,
$$
   d_H ( W(X_N),  D(0,\sqrt{2}\lambda )) \to 0 \quad \mbox{ and } \quad
   d_H ( W(Y_N),  D(0, \lambda ))\to 0 .
$$
In particular with probability one,
$$
   d_H ( W(G_N),  D(0,\sqrt{2})) \to 0 , \quad
   d_H ( W(G^{\R}_N),  D(0,\sqrt{2})) \to 0
$$
and
$$
   d_H ( W(T_N),  D(0,\sqrt{2})) \to  0.
$$
\end{theorem}

\begin{proof}
It is easy to check that the  entries of $\sqrt{N} \, Re (e^{i\theta} X_N)$  satisfy
conditions of Theorem~5.2 in \cite{BS}, that is the diagonal entries are i.i.d.
real random variables with finite second moment; the above diagonal entries
are i.i.d. mean zero complex variables with finite fourth moment and of variance
$\lambda ^2/2$. Therefore, Theorem~5.2 in \cite{BS} implies that
$\| Re (e^{i\theta} X_N)\| \to \sqrt{2} \lambda$. Theorem~\ref{mega} applied with
$R=\sqrt{2}\lambda$ provides the first limit. For the triangular matrix $Y_N$ the proof
is the same, we just need to note that the above diagonal entries of
$\sqrt{N} \, Re (e^{i\theta} Y_N)$ have variances $(\lambda/2)^2$.
The ``in particular" part follows immediately.
%
%
\end{proof}

\medskip

We now turn to quantitative estimates for ensembles $G_N$ and $T_N$.

\medskip

\begin{theorem}\label{ginib}
There exist absolute positive constants $c$ and $C$ such that for every $\eps\in (0, 1]$ and every $N$,
$$
  \PP\left( d_H\left(W(G_N), D(0,\sqrt{2})\right) \geq \eps \r)  \leq C \ \eps^{-2} \ \exp(-c N \eps ^{3}).
$$
\end{theorem}

\medskip
\noindent
{\bf Remark 1.} Note that by Borel-Cantelli lemma,
this theorem also implies that $d_H ( W(G_N),  D(0,\sqrt{2})) \to 0$.

\medskip

\begin{proof}
Note that for every real $\theta$ the  distributions of $G_N$
and $e^{i\theta } G_N$ coincide.
Note also that $\mbox{Re } (G_N)$ is a $1/\sqrt{2N}$ multiple of a GUE.
Thus, the desired result follows from the quantitative
part of Theorem~\ref{mega}  by (\ref{maxev}) and (\ref{norm})
(and by adjusting absolute constants).
\end{proof}

\medskip
\noindent
{\bf Remark 2.}
It is possible to establish a direct link between Theorem~\ref{ginib},
geometry of the set of mixed quantum states
and the Dvoretzky theorem \cite{Dv61, MiS}.

As before, let ${ Q}_N=\{\varrho: \varrho=\varrho^*, \ \varrho\ge 0, \ {\rm Tr}\varrho=1 \}$
be the set of complex density matrices of size $N$.
It is well known \cite{BZ06} that working in the geometry induced by the
Hilbert-Schmidt distance this set of (real) dimension $N^2 - 1$ is
inscribed inside a sphere of radius $\sqrt{(N-1)/N}\approx 1$,
and it contains a ball of radius  $1/\sqrt{(N-1)N}\approx 1/N$.
Applying the Dvoretzky theorem and the techniques of \cite{AS06}, one can prove
the following result \cite{AS12}: for large $N$
a generic two-dimensional projection of the set  ${Q}_N$
is very close to the Euclidean disk of radius $r_N=2/\sqrt{N}$.
Loosely speaking, in high dimensions a typical projection of a convex body
becomes close to a circular disk -- see e.g. \cite{AL}.

To demonstrate a relation with the numerical range of random matrices we apply
results from  \cite{DGH+11}, where it was shown that for any matrix $X$ of order $N$
 its numerical range $W(X)$ is  up to a translation and dilation equal
 to an orthogonal projection of the set ${Q}_N$.
The matrix $X$ determines the projection plane, while
the scaling factor for a traceless matrix reads
$\alpha (X) = \sqrt{\frac{1}{2}({\rm Tr} XX^* + | {\rm Tr} X^2|)}$.

Complex Ginibre matrices are asymptotically traceless,
and the second term $|{\rm Tr} G^2|$ tends to zero,
so the normalization condition used in this work, $\E {\rm Tr} G_NG_N^*=N$,
implies that $\E \alpha(G_N)$ converges asymptotically to $\sqrt{N/2}$.
It is natural to expect that
the projection of ${Q}_N$  associated with the complex Ginibre
matrix $G_N$ is generic and is characterized by the Dvoretzky theorem.

Our result shows that the random projection  of ${Q}_N$,
associated with the complex Ginibre matrix $G_N$
does indeed have the features expected
in view of Dvoretzky's theorem
and is close to a disk of radius $r_N  \E \alpha(G_N) = \sqrt{2}$.

\medskip


\begin{theorem}\label{triang}
There exist absolute positive constants $c$ and $C$ such that for every $\eps\in (0, 1]$ and every $N$,
$$
  \PP\left( d_H\left(W(T_N), D(0,\sqrt{2})\right) \geq \eps \r)  \leq C \ \eps^{-2} \ \exp(-c N \eps ^{3}).
$$
\end{theorem}

\medskip
\noindent
{\bf Remark 3. } Note that by Borel-Cantelli lemma,
this theorem also implies that $d_H ( W(T_N),  D(0,\sqrt{2})) \to 0$.

\medskip

\begin{proof}
Note that for every real $\theta$ the  distributions of $T_N$
and $e^{i\theta } T_N$ coincide. As was mentioned above
$\mbox{Re } T_N$ can be presented as $Z_N/\sqrt{2(N-1)}$, where
 $Z_N$ is a complex Hermitian $N\times N$ matrix with zero on the diagonal and
independent complex Gaussian random variables of variance one above the diagonal.
Thus, by  (\ref{devtrig}), for every $\theta \in \R$ and $\eps \in (0, 1]$
$$
  \PP \left( \left|  \|  \mbox{Re } (e^{i\theta } T_N)\| - \sqrt{2} \right| > \eps \r) \leq
  C \exp(-c N \eps ^{3/2})
$$
(one needs to adjust the absolute constants).
Since $X_N= \mbox{Re } X_N + i \mbox{Im } X_N = \mbox{Re } X_N + i \mbox{Re }(e^{-i\pi/2} X_N )$,
$$
  \PP \left( \| T_N\|  \geq  3  \r) \leq  C_2 \exp(-c_1 N) .
$$
Thus, applying Theorem~\ref{mega} (with $R=\sqrt{2}$ and $A=3$),
we obtain the desired result.
\end{proof}

\smallskip

\section{Further extensions.}
\label{concrem}

Note that the first part of the proof of Theorem~\ref{mega},  the inclusion of $W(X_N)$
into the disk, can be extended to a more general setting, when $R$ is not a constant but
a function of $\theta$. Namely, let $R : \R \to [1, \infty)$ be a $(2\pi)$-periodic
continuous  function. Let $K(R)$ be defined by (\ref{radial}), i.e.
$$
  K(R) := \{ \lambda e^{- i\theta} \, R(\theta) \, \, : \, \, \lambda \in [0, 1],\,
  \theta \in [0, 2\pi)\}.
$$
Note that if we identify $\C$ with $\R^2$ and $\theta$ with the direction $e^{-i\theta}$
then $R$ becomes the radial function of the star-shaped body $K(R)$. Then we have the
following

\begin{theorem} \label{inclusion} Let $K(R)$ be a star-shaped body with a continuous radial
function $R(\theta)$, $\theta\in [0, 2\pi)$.
Let $\{X_N\}_N$ be a sequence of complex random
$N\times N$  matrices such that for every $\theta\in [0, 2\pi)$  with probability one
$$
     \lim _{N\to \infty} \|\mbox{\rm Re }  (e^{i\theta} X_N)\|   = R(\theta).
$$
Then with probability one
$$
  \lim _{N\to \infty} d_H ( W(X_N)\setminus K(R), \emptyset) =0
$$
(in other words asymptotically the numerical range is contained in $K(R)$).
Furthermore, if there exists $A> 0$ such that for every $N\geq 1$,
$$
  \PP\left( \|X_N\| > A \right) \leq p_N
$$
and for every $\eps\in (0,1/2)$, $N\geq 1$, $\theta\in \R$,
$$
   \PP\left( \left|\, \|\mbox{\rm Re }  (e^{i\theta} X_N)\| - R(\theta) \right| > \eps \right) \leq q_N (\eps)
$$
then  for every
$\eps \in (0, 1/2)$ and every $N$ one has
$$
  \PP\left( d_H(W(X_N)\subset K(R + (2A + 1) \eps)  \r)  \leq
   p_N + 2 L \eps^{-1} \   q_N (\eps) ,
$$
where $L$ denotes the length of the curve $\{R(\theta)\}_{\theta\in [0, 2\pi)}$.
\end{theorem}

\medskip

\noindent
{\bf Remark 4.}
The proof below can be adjusted to prove the inclusion (\ref{inclus})
(in fact (\ref{inclus}) is simpler, since it does not require the approximation).

\smallskip

\noindent
{\bf Remark 5.}
 Under assumptions of Proposition~\ref{inclusion} on the convergence of
norms to $R$, the function $R$ must be continuous. Indeed, for
every $\theta$ and $t$ one has with probability one
$$
  |R(\theta) -R(t) | \leq \lim _{N\to\infty }
   \left| \, \|\mbox{Re }  (e^{i\theta} X_N)\|  - \|\mbox{Re }  (e^{i t} X_N)\| \right|
$$
$$
  \leq  \left| e^{i\theta} - e^{i t} \right|  \limsup _{N\to\infty }
  \| X_N\|
$$
and
$$
 \limsup _{N\to \infty} \|X_N\| \leq \limsup _{N\to \infty} \|\mbox{Re } X_N\|
  + \limsup _{N\to \infty} \|\mbox{Im } X_N\| = R(0) + R(-\pi/2) .
$$

\smallskip

\noindent
{\bf Remark 6.}
Continuity and periodicity are not the only constraints that $R$
should satisfy. For Theorem \ref{inclusion} not to be an empty statement,
The set $K(R)$ should also have the property of being convex.
This is clearly a necessary condition, and it can be proved
by simple diagonal examples that it is also a sufficient condition.

\bigskip

\noindent
{\bf Proof. }
Fix $\eps \in (0,1/2)$.
Denote
$$
  \lam(\theta, N) := \sup\{ \mbox{Re } (e^{i\theta} X_N y, y), \|y\|_2=1\}.
$$
Note that
$$
  \lam(\theta, N) = \|\mbox{Re }  (e^{i\theta} X_N)\|.
$$
Thus for every $\theta \in \R$ with probability one
$$
   \lim _{N\to \infty } \lam(\theta, N) = R(\theta) .
$$

Let $\partial K =\{R(\theta) \, \mid \, \theta\in [0, 2\pi) \}$ denote the boundary of  $K(R)$.
Choose a finite set   $\mathcal{N}$ in $[0, 2\pi]$ so that $\{R(\theta) e^{i\theta} \}_{\theta\in \mathcal{N}}$
is an $\eps$-net  in $\partial K$
(in the Euclidean metric).
Then, with probability one, for every $\theta \in {\mathcal{N}}$ one has
$\lam(\theta, N)\to R(\theta)$.

As before, note
$$
 \max_{\theta} R(\theta) \leq  \limsup _{N\to \infty} \|X_N\| \leq   R(0) + R(-\pi/2) .
$$
Choose $A\geq 1$ and  $N\geq 1$ such that  for every $M\geq N$ one has
$$
  \|X_M\| \leq  A \quad \mbox{ and } \quad \forall  \theta \in {\mathcal{N}} \, \, \, \,
   |\lam(\theta, M) - R(\theta)| \leq \eps .
$$
Note that the supremum in the definition of $\lam(\theta, N)$ is attained
and that
$$
  | e^{i\theta } - e^{i t }| \, \|X_N\| \leq
  | e^{i\theta } - e^{i t }| \,   A,
$$
whenever $\|y\|_2=1$. As was mentioned in the remark following the theorem,
$$
   |R(\theta) - R(t) | \leq  \left| e^{i\theta} - e^{i t} \right| A .
$$
Therefore, using approximation by elements of $\mathcal{N}$ and the simple estimate
$\left| e^{i\theta} - e^{i t} \right|\leq  \eps$, whenever $|\theta - t|\leq \eps$,
we obtain that for every real $t$ one has
\begin{align}\label{approx} \nonumber
  |\lam(t, N) - R(t) | \\ \nonumber
  &\leq |\lam(t, N) - \lam(\theta, N)| + |\lam(\theta, N) - R(\theta) |
 + |R(\theta) - R(t) | \\  &\leq  (2 A + 1) \eps .
\end{align}

Now let $y_0$ of norm one be such that  $(X_N y_0, y_0)$ is in the direction $e^{it}$, that is
$(X_N y_0, y_0)  = e^{it} R$ for some real positive $R$. Then
$$
  R = e^{-i t} (X_N y_0, y_0) = \mbox{Re } (e^{-i t}  X_N y_0, y_0) \leq  \lam (-t, N) \leq R(-t) +
  (2 A + 1) \eps  .
$$
This shows that $W(X_N) \subset K(R + (2 A + 1) \eps)$.

The quantitative estimates are obtained in the same way as in the proof of Theorem~\ref{mega}.
\qed

\medskip

As an example consider the following matrix. Let $H_1$, $H_2$ be independent distributed as $G_N$,
$a, b >0$ and $A:= a H_1 + i b H_2$. Then it is easy to see that $\mbox{Re } (e^{i\theta} A)$ is
distributed as $r(\theta) G_N$, where $r(\theta)=\sqrt{a^2 \cos^2 \theta + b^2 \sin ^2 \theta}$.
Therefore $\| \mbox{Re } (e^{i\theta} A)\| \to R(\theta):= \sqrt{2} r(\theta)$. Theorem~\ref{inclusion}
implies that $W(A)$ is asymptotically contained in $K(R)$ which is an ellipse.

\section{Norm estimate for the upper triangular matrix}
\label{secttri}

In this section we prove that $\|T_N\| \to \sqrt{2 e}$, as claimed in
Eq. \eqref{trnorm}
of the introduction  (Theorem~\ref{trinor}).
For the purpose of this section it is convenient to renormalize the
matrix $T_N$ and to consider $\bar T_N$, which is strictly upper diagonal
and whose entries above the diagonal
are complex centered i.i.d. Gaussians of variance $1/\sqrt{N}$.
Thus, $(\bar T_N)_{ij} = \sqrt{(N-1)/(2N)}  T_{ij}$.

We also consider upper triangular matrices $T_N'$, whose entries
above and on the diagonal are complex centered i.i.d. Gaussians of variance
$1/\sqrt{N}$. Note that $\bar T_N$ and $T_N'$ differs on the diagonal only,
therefore the following lemma follows from (\ref{maxgauss}).

\begin{lemma}
The operator norm of $\bar T_N$ converges with probability one to a limit $L$
iff the operator norm of $T_N'$ converges with probability one to $L$.
\end{lemma}

We reformulate the limiting behavior of $\| T_N\|$
in terms of $\bar T_N$. We prove the following theorem, which is clearly equivalent to (\ref{trnorm}).

\begin{theorem}\label{trinor}
With probability one, the operator norm of the sequence of random matrices
$\bar T_N$ tends to $\sqrt{e}$.
\end{theorem}

Let us first recall the following theorem, proved in \cite{DyH}.
\begin{proposition}
For any integer $\ell$,
$$\lim_N \E (N^{-1}Tr ((\bar T_N \bar T_N^*)^\ell)=\frac{\ell^\ell}{(\ell+1)!}$$
\end{proposition}


We will use the following auxiliary constructions.
Fix a positive integer parameter $k$, denote $m=[N/k]$ (the largest integer
not exceeding $N/k$), and define the  upper triangular matrix
$\bar T_{N,k}$ as follows: $(\bar T_{N,k})_{i, j} =0$ if $\ell m +1
\leq j\leq (\ell +1) m$ and $i\geq \ell m+1$ for some $\ell \geq 0$, and
$(\bar T_{N,k})_{i, j} =(\bar T_{N})_{i, j}$ otherwise. In other words
we set more entries to be equal to 0 and we have either $k\times k$ or
$(k+1)\times (k+1)$ block strictly triangular matrix (if $N$ is not multiple
of $k$ then  the last, $(k+1)$th, ``block-row"  and ``block-column"
have either their number of rows or columns strictly less than $N/k$).

\smallskip

We start with the following

\begin{lemma}\label{def-fk} Let $k$ be a positive integer and  $N$ be a multiple of $k$.
Then with probability one, $\|\bar T_{N,k}\|$ converges to a quantity $f_k$ as $N\to\infty$.
\end{lemma}

\begin{proof}
Note that the complex Ginibre matrix is, up to a proper normalization, distributed
as $A_1 + i A_2$, where $A_1$ and $A_2$ are i.i.d. GUE. Thus, when
 $N$ is a multiple of $k$,  $\bar T_{N,k}$ can be seen
as a $k\times k$ block matrix of $N/k\times N/k$ matrices, which are
linear combinations of i.i.d. copies of GUE. A Haagerup-Thorbjornsen result \cite{HTh}
ensures convergence with probability one of the norm.
\end{proof}

At this point it is not possible to compute $f_k$ explicitly.
Actually it will be enough for us
to understand the asymptotics of $f_k$ as $k\to\infty$.

\smallskip

In the next lemma, we remove the condition that $N$ be a multiple of $k$.

\begin{lemma}
 Let $k$ be a positive integer.
Then with probability one, $\| \bar  T_{N,k}\|$ converges to
to the quantity $f_k$ defined in Lemma \ref{def-fk} as $N\to\infty$.
\end{lemma}

\begin{proof}
Let $N\geq k$. Denote by $N_+$ the first multiple of $k$ after $N$.
Up to an overall multiple $N/N_+$ (imposed by the normalization
that is dimension dependent), we can realize $\bar T_{N,k}$ as a compression of
$\bar T_{N_+,k}$.
Since a compression reduces the operator norm, thanks to the previous lemma, we have with probability one,
$$\limsup_{N\to \infty} \|\bar T_{N,k}\|\leq f_k.$$
Similarly, by $N_-$ denote the first multiple of $k$ before $N$.
Up to an overall multiple $N_-/N$, we can realize $\bar T_{N_-,k}$ as a compression of
$\bar T_{N,k}$. Therefore we have with probability one,
$$\liminf_{N\to \infty} \|\bar T_{N,k}\|\geq f_k.$$
These two estimates imply the lemma.
\end{proof}

In the next Lemma, we compare the norm of $\bar T_{N,k}$ with  the norm of $\bar T_N$.

\begin{lemma}\label{seqtozer}
With probability one for every $k$ we have
$$\limsup_{N\to \infty} \left|  \|\bar T_{N,k}\|-\|\bar T_N\| \r| \leq 3/\sqrt{k}.$$
\end{lemma}

\begin{proof} For every fixed $k\leq N$ we consider
a matrix $D_{N,k}$ distributed as $\bar T_{N,k}-\bar T_N$.
Setting as before $m=[N/k]$, the entries of  $D_{N,k}$ are i.i.d. Gaussian
of variance $1/\sqrt{N}$ if  $\ell m +1 \leq j\leq (\ell +1) m$ and $i\geq \ell m+1$
for some $\ell \geq 0$, and $(D_{N,k})_{i, j} =0$  otherwise.
Clearly, this matrix is diagonal by block. It consists of $k$ diagonal
 blocks of $m  \times m$ strictly upper triangular random matrices with
entries  of variance $1/N$ and possibly one more block of smaller size.

Let us first work on estimating the tail of the operator norm on a diagonal block of size
$m  \times m$, which will be denoted by $X_N$.
It follows
directly from the  Wick formula that the quantities $\E(Tr((X_N X_N^*)^{\ell}))$ are bounded above
by quantities $\E(Tr((\tilde X_N\tilde X_N^*)^{\ell}))$, where $\tilde X_N$ is the same matrix as $X_N$
without the assumption that lower triangular entries are zero (in other words, it is a rescaled
complex Ginibre matrix of size $m  \times m$).
From there, we can make  estimates following arguments \`a la Soshnikov \cite{sos1}
and obtain that the tail of the operator norm of $X_N$ is majorized by the tail of
the operator norm of $\tilde X_N$.
More precisely, we can show that there exists a constant $C_1>0$ such that
$\E(Tr((X_N X_N^*)^{\ell}))\leq C_1(2.8/\sqrt{k})^\ell$ for every $\ell\leq N^{1/4}$.
This implies that there exists another constant $C_2>0$ such that
$\E (Tr (D_{N,k}^\ell))\leq C_1 k(2.8/\sqrt{k})^{\ell}\leq C_2 (2.9/\sqrt{k})^{\ell}$
for all sufficiently large $\ell\leq N^{1/4}$.
Therefore we deduce by Jensen inequality that
the probability that the operator norm $D_{N,k}$  is larger than
$3/\sqrt{k}$ is bounded by $C^{-N}$ for some universal constant $C>1$.
By Borel-Cantelli lemma, with probability one we have
$$
  \limsup_{N\to \infty} \left|  \|\bar T_{N,k}\|-\|\bar T_N\| \r| \leq 3/\sqrt{k}.
$$
The result follows by the triangle inequality.
\end{proof}

As a consequence we obtain the  following lemma.

\begin{lemma}
The sequence $f_k$ converges to some constant $f$ as $k\to\infty$
and  $\|\bar T_N\|$ converges to $f$ with probability one.
\end{lemma}

\begin{proof}
By Lemma~\ref{seqtozer} and the triangle inequality, we get that with probability one,
$$\limsup_{N\to \infty} \left| \|\bar  T_{N,k_1}\|-\|\bar T_{N,k_2}\| \r|\leq 3/\sqrt{k_1}+3/\sqrt{k_2}.$$
Therefore, evaluating the limit on the left hand side, we observe that
$\{f_k\}_k$  is a Cauchy sequence. Thus it converges to a constant $f$.

Next, we see that for any $\varepsilon >0$, taking $k$ large enough,
we obtain that with probability one,
$$\limsup_{N\to \infty}  \left| \|\bar T_{N}\|-f \r|\leq \varepsilon .$$
Letting $\varepsilon\to 0$, we obtain the desired result.
\end{proof}

Now we are ready to finish the proof of Theorem~\ref{trinor}.

\begin{proof}[Proof of Theorem~\ref{trinor}]
It is enough to prove that $f=\sqrt{e}$.
It follows from \cite{HTh}  that
$$
   f_k=\lim_{\ell\to \infty} \sqrt[2\ell]{\lim_{N\to \infty}
   \E(N^{-1}Tr((\bar T_{N,k}\bar T_{N,k}^*)^{\ell}))}.
$$
Given $\ell$ and $N$, it follows  from Wick's theorem
that $\E (N^{-1}Tr((\bar T_{N,k}\bar T_{N,k}^*)^{\ell}))$
increases and converges as $k\to\infty$ pointwisely to
$\E(N^{-1}Tr((\bar T_{N}\bar T_{N}^*)^{\ell}))$.
So the same result holds if we let $N\to\infty$ (by Dini's theorem), namely
$$
  \lim_{k\to \infty} \lim_{N\to \infty} \E(N^{-1}Tr((\bar T_{N,k}\bar T_{N,k}^*)^{\ell}))
  =\lim_{N\to \infty} \E(N^{-1}Tr((\bar T_{N}\bar T_{N}^*)^{\ell})).
$$
Observing that
$$
  \sqrt[2\ell]{\lim_N \E(N^{-1}Tr((\bar T_{N,k}\bar T_{N,k}^*)^{\ell}))}
$$
increases as a function of $\ell$ and applying once more Dini's theorem,
we obtain that
$$
 \lim_{k\to \infty} f_k=\lim_{\ell\to \infty} \sqrt[2\ell]{\lim_{N\to \infty}
 \E(N^{-1}Tr((\bar T_{N}\bar T_{N}^*)^{\ell}))}.
$$
Therefore
$$
   \lim_{k\to \infty} f_k=\lim_{\ell\to \infty} \sqrt[2\ell]{\frac{\ell^{\ell}}{(\ell+1)!}}
   = \sqrt{e}
$$
by the Stirling formula. This completes the proof.
\end{proof}

\bigskip

\noindent
{\bf Acknowledgment. } We are grateful to Guillaume Aubrun and Stanis{\l}aw Szarek
for fruitful discussions on the geometry of the set of quantum states, helpful remarks,
and for letting us know about their results prior to publication.
It is also a pleasure to thank Zbigniew Pucha{\l}a and Piotr {\'S}niady
for useful comments. Finally we would like to thank an anonymous referee for careful
reading and valuable remarks which have helped us to improve the presentation.


\end{document}